\documentclass[11pt]{article}

\usepackage{amsmath}
\usepackage{amsfonts}
\usepackage{amssymb}
\usepackage{centernot}
\usepackage[margin=1in]{geometry}
\usepackage[colorlinks=true,linkcolor=blue,citecolor=blue,urlcolor=blue]{hyperref}
\usepackage{enumitem}
\usepackage{graphicx,fancybox}
\usepackage{tikz}
\usepackage{subfigure}
\usepackage{amsthm}
\usepackage[capitalise]{cleveref}
\usepackage{pgf}
\usepackage[font={small}]{caption}
\usepackage{soul}
\usepackage{hyphenat}

\newcommand{\gra}{\operatorname{gra}}
\newcommand{\zer}{\operatorname{zer}}
\newcommand{\Fix}{\operatorname{Fix}}

\newcommand{\dom}{\operatorname{dom}}
\newcommand{\Id}{\operatorname{Id}}

\newcommand{\bs}{\boldsymbol}

\def\bomega{{\boldsymbol \omega}}

\def\bC{{\boldsymbol C}}
\def\bD{{\boldsymbol D}}
\def\bF{{\boldsymbol F}}
\def\bu{{\boldsymbol u}}
\def\bv{{\boldsymbol v}}
\def\bx{{\boldsymbol x}}
\def\by{{\boldsymbol y}}

\newcommand{\Tdr}{T_{\operatorname{aDR}}}

\newcommand{\Hi}{\mathcal{H}}

\newcommand{\RR}{\mathbb{R}}

\def\tto{\rightrightarrows}
\def\wto{\rightharpoonup}
\newcommand{\seq}[2]{\left(#1\right)_{#2=0}^{\infty}}

\Crefname{fact}{Fact}{Facts}
\Crefname{enumi}{}{}

\usepackage{titlesec}
\usepackage{titling}
\usepackage{anyfontsize}
\pretitle{\vspace{-1cm}\begin{center}\fontsize{18}{28}\bfseries\sffamily}

\makeatletter
\def\th@plain{%
	\thm@notefont{}
	\itshape 
}
\def\th@definition{%
	\thm@notefont{}
	\normalfont 
}

\setlist[enumerate]{topsep=3pt}

\newtheorem{theorem}{Theorem}[section]

\newtheorem{corollary}{Corollary}[section]
\newtheorem{lemma}{Lemma}[section]
\newtheorem{fact}{Fact}[section]

\theoremstyle{definition}
\newtheorem{definition}{Definition}[section]

\newtheorem{remark}{Remark}[section]

\newcommand*\samethanks[1][\value{footnote}]{\footnotemark[#1]}

\def\begproof{\noindent{\it Proof}. \ignorespaces}
\def\endproof{\ensuremath{\hfill \quad \square}}

\begin{document}

\title{Demiclosedness principles\\ for generalized nonexpansive mappings}

\author{Sedi Bartz \thanks{Department of Mathematical Sciences, Kennedy College of Sciences, University of Massachusetts Lowell, MA, USA. E-mail:~\href{mailto:sedi_bartz@uml.edu}{sedi\_bartz@uml.edu}, \href{mailto:ruben_campoygarcia@uml.edu}{ruben\_campoygarcia@uml.edu}, \href{mailto:hung_phan@uml.edu}{hung\_phan@uml.edu}} 
\and Rub\'en Campoy \samethanks[1] \and Hung M.\ Phan \samethanks[1]}

\maketitle

\begin{abstract}
Demiclosedness principles are powerful tools in the study of convergence of iterative methods. For instance, a multi-operator demiclosedness principle for firmly nonexpansive mappings is useful in obtaining simple and transparent arguments for the weak convergence of the shadow sequence generated by the Douglas--Rachford algorithm. We provide extensions of this principle which are compatible with the framework of more general families of mappings such as cocoercive and conically averaged mappings. As an application, we derive the weak convergence of the shadow sequence generated by the adaptive Douglas--Rachford algorithm.

\end{abstract}

\section{Introduction}

Demiclosedness principles play an important role in convergence analysis of fixed point algorithms. The concept of \emph{demiclosedness} sheds light on topological properties of mappings, in particular, in the case where a weak topology is considered. More precisely, given a weakly sequentially closed subset $D$ of a Hilbert space $\Hi$, the mapping $T:D\to\Hi$ is said to be \emph{demiclosed} at $x\in D$, if for every sequence $(x_k)$ in $D$ such that $(x_k)$ converges weakly to $x$ and $T(x_k)$ converges strongly, say, to $u$, it follows that $T(x)=u$. By its definition, demiclosedness holds trivially whenever $T$ is weakly sequentially continuous; however, it does not hold in general. Let $\Id$ denote the identity mapping on $\Hi$. A fundamental result in the theory of nonexpansive mappings is Browder's celebrated \emph{demiclosedness principle}~\cite{Browder}, which asserts that, if $T$ is nonexpansive, then the mapping $\Id-T$ is demiclosed at every point in $D$. Browder's result holds in more general settings and, by now,  has become a key tool in the study of  asymptotic and ergodic properties of nonexpansive mappings; see~\cite{DC1,DC2,DC3,DC4}, for example.

In~\cite{Demiclosed}, Browder's demiclosedness principle was extended and a version for finitely many firmly nonexpansive mappings was provided. As an application, a simple proof of the weak convergence of the \emph{Douglas--Rachford (DR) algorithm}~\cite{DR56,LM79} was also provided in~\cite{Demiclosed}: The DR algorithm belongs to the class of splitting methods for the problem of finding a {\em zero of the sum} of two maximally monotone operators $A,B:\Hi\tto\Hi$, see \eqref{e:0inA+B}. The DR algorithm generates a sequence by an iterative application of the  DR operator (see~\eqref{eq:aDR_opeartor},~\eqref{e:TaDR} and the comment thereafter), which can be expressed in terms of the resolvents (see~\cref{d:resol}) of $A$ and $B$. If the solution set is nonempty, then the DR sequence converges weakly to a fixed point such that the resolvent of $A$ maps it to a zero of $A+B$. Thus, we see that, in fact, we are interested in the image of the DR sequence under the resolvent of $A$.
This image is often referred to as the \emph{shadow sequence}. The resolvent of a maximally monotone operator is continuous (in fact, firmly nonexpansive) but not weakly continuous, in general. Hence, the convergence of the shadow sequence can not be derived directly from the convergence of the DR sequence, unless the latter converges in norm. However, in general, norm convergence does not hold: In~\cite{weakDR}, an example of a DR iteration which does not converge in norm was explicitly constructed. Regardless of this fact, the weak convergence of the shadow sequence was established by Svaiter in~\cite{Svaiter}. A simpler and more accessible proof of  the weak convergence of the shadow sequence was later given in~\cite{Demiclosed} by employing a multi-operator demiclosedness principle. A demiclosedness principle for circumcenter mappings, a class of operators that is generally not continuous, was recently developed in \cite{BOW18}.

In this paper, we present an extended demiclosedness principles for more general families of operators, which are not necessarily firmly nonexpansive, provided that they satisfy a (firm) nonexpansiveness {balance condition}. We are motivated by the \emph{adaptive Douglas--Rachford (aDR) algorithm} which was recently studied in~\cite{DPadapt} in order to find a zero of the sum of a weakly monotone operator and a strongly monotone operator. Furthermore, the framework of~\cite{DPadapt}  has been recently extended in~\cite{BDP19} in order to hold for monotonicity and comonotonicity settings as well. In both studies~\cite{DPadapt,BDP19}, the convergence of the shadow sequence generated by the aDR is guaranteed only under the assumption that the sum of the operators is strongly monotone. Moreover, the corresponding resolvents in the aDR are not necessarily firmly nonexpansive and, consequently, the demiclosedness principles of~\cite{Demiclosed} can not be directly applied in this framework. Our current approach is compatible with the framework of the aDR. Consequently, we employ our generalized demiclosedness priciples in order to obtain weak convergence of the shadow sequence of the aDR in most cases. To this end we employ and extend techniques and results from~\cite{Demiclosed}.

The remainder of this paper is organized as follows: In~\Cref{sec:prelim}, we review preliminaries and basic results. New demiclosedness principles are provided in~\Cref{sec:demiclosed}. In~\Cref{sec:aDR}, we employ the demiclosedness principles from~\Cref{sec:demiclosed} in order to obtain the weak convergence of the shadow sequence of the adaptive Douglas--Rachford algorithm. Finally, in~\Cref{sec:conclu} we conclude our discussion.

\section{Preliminaries}\label{sec:prelim}

Throughout this paper, $\Hi$ is a real Hilbert space equipped with inner product $\langle\cdot , \cdot\rangle$ and induced norm $\|\cdot\|$. The weak convergence and strong convergence are denoted by $\wto$ and $\to$, respectively. We set $\RR_+:=\{r\in\RR: r\geq 0\}$ and $\RR_{++}:=\{r\in\RR: r>0\}$. Given a set-valued operator $A:\Hi\tto\Hi$, the \emph{graph}, the \emph{domain}, the set of \emph{fixed points} and the set of \emph{zeros} of A, are denoted, respectively, by $\gra A$, $\dom A$, $\Fix A$ and $\zer A$; i.e.,
\[
\begin{aligned}
\gra A&:=\big\{(x,u)\in\Hi\times\Hi : u\in A(x)\big\}, &\dom A&:=\big\{x\in\Hi : A(x)\neq\varnothing\big\},\\
\Fix A&:=\big\{x\in\Hi : x\in A(x)\big\}
&\text{and} \quad \zer A&:=\big\{x\in\Hi : 0\in A(x)\big\}.
\end{aligned}
\]
The \emph{identity} mapping is denoted by $\Id$ and the \emph{inverse} of $A$ is denoted by $A^{-1}$, i.e., $\gra A^{-1}:=\{(u,x)\in \Hi \times \Hi : u\in A(x)\}$.

\begin{definition}\label{def:nonexp}
	Let $D\subseteq\Hi$ be nonempty, let $T:D\to\Hi$ be a mapping, set $\tau>0$ and $\theta>0$. The mapping $T$ is said to be
	\begin{enumerate}[label={\rm(\roman*)}]
		\item \emph{nonexpansive}, if \label{def:nonexp_nonexp}
		\begin{equation*}
		\big\|T(x)-T(y)\big\|\leq\|x-y\|, \quad \forall x,y\in D;
		\end{equation*}
		\item \emph{firmly nonexpansive}, if \label{def:nonexp_firm}
		\begin{equation*}
		\big\|T(x)-T(y)\big\|^2+\big\|(\Id-T)(x)-(\Id-T)(y)\big\|^2\leq\|x-y\|^2, \quad \forall x,y\in D,
		\end{equation*}
		equivalently,
		\begin{equation*}
		\big\langle x-y,T(x)-T(y)\big\rangle\geq\big\|T(x)-T(y)\big\|^2, \quad \forall x,y\in D;
		\end{equation*}
		\item $\tau$-\emph{cocoercive}, if $\tau T$ is firmly nonexpansive, i.e., \label{def:nonexp_coco}
		\begin{equation*}
		\big\langle x-y,T(x)-T(y)\big\rangle\geq \tau\big\|T(x)-T(y)\big\|^2, \quad \forall x,y\in D;
		\end{equation*}
		\item \emph{conically $\theta$-averaged}, if there exists a nonexpansive operator $R:D\to\Hi$ such that \label{def:nonexp_conical}
		\begin{equation*}
		T=(1-\theta)\Id+\theta R.
		\end{equation*}
	\end{enumerate}
\end{definition}

Conically $\theta$-averaged mappings, introduced in~\cite{BMW19} and originally named \emph{conically nonexpansive mappings}, are natural extensions of the classical $\theta$-averaged mappings; more precisely, a conically $\theta$-averaged mapping is $\theta$-averaged whenever $\theta\in{]0,1[}$. Additional properties and further discussions of conically averaged mappings, such as the following result, are available in~\cite{BDP19,GW19}.

\begin{fact}\label{fact:conical_equiv}
Let $D\subseteq\Hi$ be nonempty, let $T:D\rightarrow\Hi$ and let $\theta,\sigma>0$. Then the following assertions are equivalent:
\begin{enumerate}[label={\rm(\roman*)}]
\item $T$ is conically $\theta$-averaged;\label{fact:conical_equivI}
\item $(1-\sigma)\Id+\sigma T$ is conically $\sigma\theta$-averaged;\label{fact:conical_equivII}
\item For all $x,y\in D$,\label{fact:conical_equivIII}
$$\big\|T(x)-T(y)\big\|^2\leq \|x-y\|^2-\frac{1-\theta}{\theta}\big\|(\Id-T)(x)-(\Id-T)(y)\big\|^2.$$
\end{enumerate}
\end{fact}
\begin{proof}
See~\cite[Proposition~2.2]{BDP19}.
\end{proof}

\begin{lemma}\label{l:constants}
Let $D\subseteq\Hi$ be nonempty, let $T:D\to\Hi$ and let $\tau,\theta>0$.
\begin{enumerate}[label={\rm(\roman*)}]
\item If $T$ is $\tau$-cocoercive, then it is $\tau'$-cocoercive for any $\tau'\in{]0,\tau]}.$\label{l:constants_coco}
\item If $T$ is conically $\theta$-averaged, then it is conically $\theta'$-averaged for any $\theta'\in{[\theta,\infty[}$.\label{l:constants_conical}
\end{enumerate}
\end{lemma}
\begin{proof}
\Cref{l:constants_coco}: Follows immediately from the definition of cocoercivity. \Cref{l:constants_conical}: Follows from the equivalence \ref{fact:conical_equivI} $\iff$ \ref{fact:conical_equivIII} in \Cref{fact:conical_equiv}.  
\end{proof}

\begin{remark} We note that cocoercivity and conical averagedness generalize the notion of firm nonexpansiveness as follows:\label{rem:constants}
\begin{enumerate}[label={\rm(\roman*)}]
\item By Definition~\ref{def:nonexp}\ref{def:nonexp_firm} and~\ref{def:nonexp_coco}, the mapping $T$ is firmly nonexpansive if and only if $T$ is $1$-cocoercive. Consequently, by employing \Cref{l:constants}\Cref{l:constants_coco}, we see that whenever $\tau\geq 1$, a $\tau$-cocoercive mapping is firmly nonexpansive.\label{rem:constants_coco}
\item Similarly, the mapping $T$ is firmly nonexpansive if and only if it is conically $\frac{1}{2}$-averaged. Consequently, by employing \Cref{l:constants}\Cref{l:constants_conical}, we see that whenever $\theta\leq\frac{1}{2}$, a conically $\theta$-averaged mapping is firmly nonexpansive.\label{rem:constants_coni}
\end{enumerate}

\end{remark}

In our study we will employ the following generalized notions of monotonicity.

\begin{definition}\label{def:monotone}
Let $A:\Hi\tto\Hi$ and let $\alpha\in\RR$. Then $A$ is said to be
\begin{enumerate}[label={\rm(\roman*)}]

\item \emph{$\alpha$-monotone}, if \label{def:monotone_mono}
\begin{equation*}
\langle x-y,u-v\rangle\geq \alpha\|x-y\|^2,\quad\forall (x,u),(y,v)\in\gra A;
\end{equation*}
\item \emph{$\alpha$-comonotone}, if $A^{-1}$ is $\alpha$-monotone, i.e., \label{def:monotone_comono}
\begin{equation*}
\langle x-y,u-v\rangle\geq \alpha\|u-v\|^2,\quad\forall (x,u),(y,v)\in\gra A.
\end{equation*}
\end{enumerate}
The $\alpha$-monotone operator $A$ is said to be \emph{maximally $\alpha$-monotone} (resp. \emph{maximally $\alpha$\hyp{}comonotone}), if there is no $\alpha$-monotone (resp. $\alpha$-comonotone) operator $B:\Hi\tto\Hi$ such that $\gra A$ is properly contained in $\gra B$.
\end{definition}

\begin{remark}
Common notions of monotonicity are related to the notions in \Cref{def:monotone} as follows:
\begin{itemize}
\item In the case where $\alpha=0$, $0$-monotonicity and $0$-comonotonicity simply mean \emph{monotonicity} (see, for example, \cite[Definition~20.1]{BC17}). 
\item In the case where $\alpha>0$, $\alpha$-monotonicity is also known as \emph{$\alpha$-strong monotonicity} (see, for example, \cite[Definition~22.1(iv)]{BC17}). Similarly, $\alpha$\hyp{}comonotonicity is $\alpha$-cocoercivity in \Cref{def:nonexp}\Cref{def:nonexp_coco}.
\item In the case where $\alpha<0$, $\alpha$-monotonicity and $\alpha$-comonotonicity are also known as \emph{$\alpha$-hypomonotonicity} and \emph{$\alpha$-cohypomonotonicity}, respectively (see, for example, \cite[Definition 2.2]{CP04}). In addition, $\alpha$-monotonicity is referred to as \emph{$\alpha$-weak monotonicity} in \cite{DPadapt}.
\end{itemize}
\end{remark}

We continue our preliminary discussion by recalling the definition of the resolvent.
\begin{definition}\label{d:resol}
Let $A:\Hi\tto\Hi$. The \emph{resolvent} of $A$ is the operator defined by
\begin{equation*}
J_A:=(\Id+A)^{-1}.
\end{equation*}
The \emph{relaxed resolvent} of $A$ with parameter $\lambda>0$ is defined by
\begin{equation*}
J^{\lambda}_A:=(1-\lambda)\Id + \lambda J_A.
\end{equation*}
When $\lambda=2$, we set $R_A:=J^2_A=2J_A-\Id$, also known as the \emph{reflected resolvent} of $A$.
\end{definition}

We conclude our preliminary discussion by relating monotonicity and comonotonicity properties with corresponding properties for resolvents by recalling the following facts.

\begin{fact}[resolvents of monotone operators]\label{fact:resolvent_mono}
Let $A:\Hi\tto\Hi$ be $\alpha$\hyp{}monotone, where $\alpha\in\RR$. If $\gamma>0$ is such that $1+\gamma\alpha>0$, then 
\begin{enumerate}[label={\rm(\roman*)}]
\item $J_{\gamma A}$ is single-valued and $(1+\gamma\alpha)$-cocoercive; \label{fact:resolvent_mono_I}
\item $\dom J_{\gamma A}=\Hi$ if and only if $A$ is maximally $\alpha$-monotone. \label{fact:resolvent_mono_II}
\end{enumerate}
\end{fact}
\begin{proof}
See~\cite[Lemma~3.3(ii) and Proposition~3.4]{DPadapt}.
\end{proof}

\begin{fact}[resolvents of comonotone operators]\label{fact:resolvent_comono}
Let $A:\Hi\tto\Hi$ be $\alpha$\hyp{}comonotone, where $\alpha\in\RR$. If $\gamma>0$ is such that $\gamma+\alpha>0$, then 
\begin{enumerate}[label={\rm(\roman*)}]
\item $J_{\gamma A}$ is at most single-valued and conically $\frac{\gamma}{2(\gamma+\alpha)}$-averaged; \label{fact:resolvent_comono_I}
\item $\dom J_{\gamma A}=\Hi$ if and only if $A$ is maximally $\alpha$-comonotone. \label{fact:resolvent_comono_II}
\end{enumerate}
\end{fact}
\begin{proof}
See~\cite[Propositions~3.7 and 3.8(i)]{BDP19}.
\end{proof}

\section{New demiclosedness principles for generalized nonexpansive mappings}\label{sec:demiclosed}

In this section, we extend the multi-operator demiclosedness principles for firmly nonexpansive mappings from~\cite{Demiclosed}  to more general families of operators. To this end, we employ and extend techniques and results from~\cite{Demiclosed} in a weighted inner product space. We begin our discussion by recalling the following fact.

\begin{fact}\label{fact:Demiclosed}
Let $F:\Hi\to\Hi$ be a firmly nonexpansive mapping, let $\seq{x_k}{k}$ be a sequence in $\Hi$, and let $C,D\subseteq\Hi$ be closed affine subspaces such that {$C-C=(D-D)^\perp$}. Suppose that
\begingroup
\allowdisplaybreaks
\begin{subequations}\label{eq:dc_conditions}
 	\begin{align}
 	x_k\wto x,& \label{eq:dc1} \\
 	F(x_k)\wto y,& \label{eq:dc2} \\
 	F(x_k)-P_C(F(x_k))\to 0,& \label{eq:dc3} \\
 	(x_k-F(x_k))-P_D(x_k-F(x_k))\to 0.& \label{eq:dc4}
 	\end{align}
\end{subequations}
\endgroup
Then $y\in C$, $x\in y+D$, and $y=F(x)$.
\end{fact}
\begin{proof}
See~\cite[Corollary~2.7]{Demiclosed}.
\end{proof}

\subsection{Demiclosedness principles for cocoercive operators}
Let $n\geq 2$ be an integer and set $\bs\omega:=(\omega_1,\omega_2,\ldots,\omega_n)\in\mathbb{R}_{++}^n$.
We equip the space
\begin{subequations}\label{e:Hn}
\begin{align}
\bs\Hi:=\Hi^n=\underbrace{\Hi\times\Hi\times \cdots\times\Hi}_{n},
\end{align}
with the weighted inner product $\langle \cdot, \cdot \rangle_{\bs\omega}$ defined by
\begin{align}
\langle \bs x, \bs y \rangle_{\bs\omega}:=  \sum_{i=1}^n \omega_i\langle x_i,y_i\rangle,\quad \forall \bs x=(x_1,x_2,\ldots,x_n), \bs y=(y_1,y_2,\ldots,y_n)\in\bs\Hi.
\end{align}
\end{subequations}
Thus, $\bs\Hi$ is a Hilbert space with the induced norm 
$\|\bs x\|_{\bomega}=\sqrt{\langle\bs x,\bs x\rangle_{\bomega}}$.

Let $\bs\tau:=(\tau_1,\tau_2,\ldots,\tau_n)\in\RR^n\smallsetminus\big\{(0,\ldots,0)\big\}$. Let $\bC\subset\bs\Hi$ be the subspace defined by 
\begin{equation*}\label{eq:Ct}
\bC:=\big\{ (\tau_1x,\tau_2x,\ldots,\tau_nx) : x\in\Hi \big\}.
\end{equation*}
In the following lemma, we provide a formula for the projector $P_{\bC}$, which will be useful later.
\begin{lemma}\label{l:ProjD}
Let $\bx:=(x_1,x_2,\ldots,x_n)\in\bs\Hi$. Then the projection of $\bx$ onto $\bC$ is 
\begin{equation}\label{eq:PCt}
P_{\bC}(\bx)=
\left(\tau_1 \overline{u},\tau_2 \overline{u},\ldots,\tau_n \overline{u}\right),
\quad\text{where~}\overline{u}:=\frac{\sum_{i=1}^n \omega_i\tau_ix_i}{\sum_{i=1}^n \omega_i\tau_i^2}.
\end{equation}
\end{lemma} 
\begproof
Fix $\bx:=(x_1,x_2,\ldots,x_n)\in\bs\Hi$. Since $\sum_{i=1}^{n}\omega_i\tau_i^2>0$, $\overline{u}$ is well defined by \eqref{eq:PCt}. Set $\by=(\tau_1 \overline{u},\tau_2 \overline{u},\ldots,\tau_n \overline{u})$. We prove that $(\bs x-\bs y)\perp\bs z$ for each $\bs z\in\bC$, consequently, $P_{\bC}(\bx) = \by$. Indeed, let $\bs z=(\tau_1 v,\tau_2 v,\ldots,\tau_n v)\in \bC$. Then
\begin{flalign*}
&&\langle \bs z, \bs x-\bs y \rangle_\bomega
&=\sum_{i=1}^{n}\omega_i\langle \tau_iv, x_i-\tau_i \overline{u}\rangle\\
&& &=\left\langle v, \sum_{i=1}^{n} \omega_i\tau_i(x_i-\tau_i \overline{u})\right\rangle\\
&& &=\left\langle v, \sum_{i=1}^{n} \omega_i\tau_i x_i- \bigg(\sum_{i=1}^{n} \omega_i\tau_i^2 \bigg)\overline{u}\right\rangle
=\langle v, 0\rangle =0.
&& \endproof
\end{flalign*}

\begin{theorem}[demiclosednes principle for cocoercive operators]\label{t:demi_coco}
Let $(\rho_1,\rho_2,\ldots,\rho_n)$, $(\tau_1,\tau_2,\ldots,\tau_n)\in\mathbb{R}_{++}^n$. For each  $i\in\{1,2,\ldots,n\}$, let $F_i:\Hi\to\Hi$ be $\tau_i$-cocoercive and let $\seq{x_{i,k}}{k}$ be a sequence in $\Hi$.
 Suppose that
\begin{subequations}\label{e:t200313a}
 	\begin{align}
 	\forall i\in \{1,2,\ldots,n\}, \qquad x_{i,k}\wto x_i,\label{e:t200313a-a} \\
 	\forall i\in \{1,2,\ldots,n\}, \qquad F_i(x_{i,k})\wto y, \label{e:t200313a-b} \\
 	\sum_{i=1}^n\rho_i(x_{i,k}-\tau_iF_i(x_{i,k}))
 	\to
 	-\left(\sum_{i=1}^n\rho_i\tau_i\right)y+\sum_{i=1}^n\rho_ix_i,&\label{e:t200313a-c} \\
 	\forall i, j\in \{1,2,\ldots,n\},\qquad F_i(x_{i,k})-F_j(x_{j,k})\to 0.& \label{e:t200313a-d}
 	\end{align}
\end{subequations}
Then $F_1(x_1)=F_2(x_2)=\cdots=F_n(x_n)=y$.
\end{theorem}
\begin{proof}
For each $i\in\{1,2,\ldots,n\}$, set 
\begin{equation*}
\omega_i:=\frac{\rho_i}{\tau_i} 
\end{equation*}
and equip $\bs\Hi$ with the inner product $\langle \cdot, \cdot \rangle_{\bs\omega}$ as in \eqref{e:Hn}.
Let $\bF:\bs\Hi\to\bs\Hi$ be the mapping defined by
\begin{equation*}
\bF(\bs z)=(\tau_1F_1(z_1),\tau_2F_2(z_2),\ldots,\tau_nF_n(z_n)), \quad \bs z=(z_1,z_2,\ldots,z_n)\in \bs\Hi.
\end{equation*}
Then for every $\bu=(u_1,u_2,\ldots,u_n),\bv=(v_1,v_2,\ldots,v_n)\in\bs\Hi$, since $F_i$ is $\tau_i$ cocoercive, it follows that
\begin{equation*}
\begin{aligned}
\big\langle\bu-\bv,\bF(\bu)-\bF(\bv)\big\rangle_\bomega
&=\sum_{i=1}^{n}\omega_i\langle u_i-v_i,\tau_i F_i(u_i) - \tau_i F_i(v_i)\rangle\\
&\geq \sum_{i=1}^{n}\omega_i\tau_i^2\|F_i(u_i) - F_i(v_i)\|^2
=\|\bF(\bu)-\bF(\bv)\|_\bomega^2,
\end{aligned}
\end{equation*}
which implies that $\bF$ is firmly nonexpansive. 
Set  $\bs x:=(x_1,x_2,\ldots,x_n),\ \bs y:=(\tau_1y,\tau_2y,\ldots,\tau_ny)$ and, for each $k=0,1,2,\ldots$, set $\bs x_k:=(x_{1,k},x_{2,k},\ldots,x_{n,k})$. Then 
\begin{equation}\label{e:t200313b}
\bx_k \wto \bx
\text{~~and~~}
\bF(\bx_k)\wto \by.
\end{equation}
Let $\bC$ and $\bD$ be the affine subspaces of $\bs\Hi$ defined by
\begin{equation*}
\bC:=\left\{ (\tau_1x,\tau_2x,\ldots,\tau_nx) : x\in\Hi \right\}
\text{~~and~~}
\bD:= \bx - \by +\bC^\perp.
\end{equation*}
Then $\bC-\bC=(\bD-\bD)^\perp$. Consequently, by employing \cref{l:ProjD}, we arrive at
\begin{equation*}
P_{\bs C}\left( \bF (\bs x_k) \right)
=(\tau_1 \overline{v}_k,\ldots, \tau_n \overline{v}_k),\quad
\text{where }
\overline{v}_k:=\frac{\sum_{i=1}^{n}\omega_i\tau_i^2 F_i(x_{i,k})}{\sum_{i=1}^{n}\omega_i\tau_i^2}
=\frac{\sum_{i=1}^{n}\rho_i\tau_i F_i(x_{i,k})}{\sum_{i=1}^{n}\rho_i\tau_i}.
\end{equation*}
Since $\overline{v}_k$ is a weighted average of the $F_i(x_{i,k})$'s, \eqref{e:t200313a-d} implies that
\begin{equation*}
\forall i\in\{1,2,\ldots,n\},\quad
F_i(x_{i,k})-\overline{v}_k\to 0\quad\text{as}\quad k\to\infty,
\end{equation*}
consequently, we conclude that  
\begin{equation}\label{e:t200313c}
\bF(\bx_k)-P_{\bs C}\left( \bF (\bs x_k) \right) \to 0.
\end{equation}
We now employ the projections
\begin{equation*}
P_{\bs C}\left(\bx_k - \bF(\bx_k)\right)
=(\tau_1\overline{u}_k,\ldots, \tau_n\overline{u}_k),\quad
\text{where~}
\overline{u}_k=\frac{\sum_{i=1}^n \rho_i\left(x_{i,k}- \tau_iF_i(x_{i,k})\right)}{\sum_{i=1}^{n}\rho_i\tau_i},
\end{equation*}
and
\begin{equation*}
P_{\bs C}\left(\bx - \by\right)
=(\tau_1\overline{u},\ldots,\tau_n\overline{u})-\by,\quad
\text{where~}
\overline{u}=\frac{\sum_{i=1}^n \rho_i x_{i}}{\sum_{i}^{n}\rho_i\tau_i}.
\end{equation*}
By invoking \eqref{e:t200313a-c}, we see that $\overline{u}_k\to -y+\overline{u}$, which, in turn, implies that
\begin{equation*}
P_{\bs C}\left(\bs x_k - \bF(\bs x_k)\right) \to P_{\bs C}\left(\bs x - \bs y\right).
\end{equation*}
Consequently,
\begin{equation}\label{e:t200313d}
\begin{aligned}
\bx_k - \bF(\bx_k)- &P_{\bD}\left( \bx_k - \bF(\bx_k) \right)\\
& = \bx_k - \bF(\bx_k)- P_{\bx - \by +\bC^\perp}\left( \bx_k - \bF(\bx_k) \right)\\
& = \bx_k - \bF(\bx_k)- \left(\bx - \by + P_{\bC^\perp}\left( \bx_k - \bF(\bx_k)-(\bx - \by ) \right)\right)\\
& = (\Id- P_{\bC^\perp})\left(\bx_k - \bF(\bx_k)\right)-(\Id- P_{\bC^\perp})\left(\bx - \by\right)\\ 
& = P_{\bC}\left(\bx_k - \bF(\bx_k)\right)-P_{\bC}\left(\bx - \by\right)\to 0.
\end{aligned}
\end{equation}
Finally, since \eqref{e:t200313b}, \eqref{e:t200313c} and \eqref{e:t200313d} satisfy \eqref{eq:dc_conditions}, we may employ \Cref{fact:Demiclosed} in order to obtain $\by = \bs{F}(\bx)$, that is,
\begin{equation*}
F_i(x_i)=y,\quad\forall i\in\{1,2,\ldots,n\},
\end{equation*}
which concludes the proof.
\end{proof}

As a consequence of \cref{t:demi_coco}, we obtain the demiclosedness principle for firmly nonexpansive operators \cite[Theorem~2.10]{Demiclosed}.
\begin{corollary}[demiclosedness principle for firmly nonexpansive operators]\label{t:demi_fne}
For each ${i\in\{1,2,\ldots,n\}}$ let $F_i:\Hi\to\Hi$ be firmly nonexpansive and let $\seq{x_{i,k}}{k}$ be a sequence in $\Hi$. Suppose further that
\begingroup
\allowdisplaybreaks
\begin{subequations}\label{eq:dc_conditions_firm}
 	\begin{align}
 	\forall i\in \{1,2,\ldots,n\}, \qquad x_{i,k}\wto x_i,\label{eq:dc1_firm} \\
 	\forall i\in \{1,2,\ldots,n\}, \qquad F_i(x_{i,k})\wto y, \label{eq:dc2_firm} \\
 	\sum_{i=1}^n(x_{i,k}-F_i(x_{i,k}))\to -ny+\sum_{i=1}^nx_i,& \label{eq:dc3_firm} \\
 	\forall i, j\in \{1,2,\ldots,n\}, \qquad F_i(x_{i,k})-F_j(x_{j,k})\to 0.& \label{eq:dc4_firm}
 	\end{align}
\end{subequations}
\endgroup
Then $F_1(x_1)=F_2(x_2)=\cdots=F_n(x_n)=y$.
\end{corollary}
\begin{proof}
The proof follows by observing that firmly nonexpansive operators are cocoercive with constant $\tau=1$ and by setting $\tau_1=\cdots=\tau_n=1$ and $\rho_1=\cdots=\rho_n=1$ in \cref{t:demi_coco}.
\end{proof}

\begin{remark} By letting $n=1$ in \cref{t:demi_fne}, we obtain a special case of \cref{fact:Demiclosed} where $C=\Hi$ and $D=\{x-y\}$, which, in turn, is equivalent to Browder's original demiclosedness principle \cite{Browder} (alternatively, see \cite[Theorem~4.27]{BC17}).
\end{remark}

\begin{remark}[\cref{t:demi_coco} versus \cref{t:demi_fne}]
\label{rem:demi_coco}
A demiclosedness principle for cocoercive mappings can be derived directly from \Cref{t:demi_fne} when we apply the latter to the firmly nonexpansive mappings $\tau_1F_1,\tau_2F_2,\dots,\tau_nF_n$. However, this does not yield \cref{t:demi_coco}: In this case the cocoercivity constants will appear in \eqref{eq:dc2_firm} and \eqref{eq:dc4_firm}, however, they are neither a part~of~\eqref{e:t200313a-b}~nor~\eqref{e:t200313a-d}.
\end{remark}

Following \Cref{rem:demi_coco}, in \Cref{t:demi_coco}, the cocoercivity constants $\tau_i$ are not a part of the conditions in \eqref{e:t200313a} except for the condition \eqref{e:t200313a-c}. In the following result, we do not incorporate cocoercivity constants in any of the convergence conditions, however, in exchange, we do impose a certain balance condition on these constants.

\begin{theorem}[demiclosedness principle for balanced cocoercive operators]\label{th:Demiclosed_coco}
For each ${i\in\{1,2,\ldots,n\}}$ let $F_i:\Hi\to\Hi$ be a $\tau_i$-cocoercive mapping where $\tau_i>0$ and let $\seq{x_{i,k}}{k}$ be a sequence in $\Hi$. Suppose that there exists $(\rho_1,\rho_2,\ldots,\rho_n)\in\mathbb{R}_{++}^n$ such that the weighted average
\begin{equation}\label{eq:taus_relation}
\frac{\sum_{i=1}^n\rho_i\tau_i}{\sum_{i=1}^n\rho_i}\geq 1,
\end{equation}
and suppose further that
\begingroup
\allowdisplaybreaks
\begin{subequations}\label{eq:dc_conditions_coco}
 	\begin{align}
 	\forall i\in \{1,2,\ldots,n\}, \qquad x_{i,k}\wto x_i,\label{eq:dc1_coco} \\
 	\forall i\in \{1,2,\ldots,n\}, \qquad F_i(x_{i,k})\wto y, \label{eq:dc2_coco} \\
 	\sum_{i=1}^n\rho_i(x_{i,k}-F_i(x_{i,k}))\to -\bigg(\sum_{i=1}^n\rho_i\bigg)y+\sum_{i=1}^n\rho_ix_i,& \label{eq:dc3_coco} \\
 	\forall i, j\in \{1,2,\ldots,n\}, \qquad F_i(x_{i,k})-F_j(x_{j,k})\to 0.& \label{eq:dc4_coco}
 	\end{align}
\end{subequations}
\endgroup
Then $F_1(x_1)=F_2(x_2)=\cdots=F_n(x_n)=y$.
\end{theorem}
\begin{proof} In view of \eqref{eq:taus_relation}, we may choose $(\tau'_1,\tau_2',\ldots,\tau_n'),\ \tau_i'\in\left]0,\tau_i\right]$ such that
\begin{equation*}\label{eq:taus_prime_relation}
\frac{\sum_{i=1}^n\rho_i\tau'_i}{\sum_{i=1}^n\rho_i}= 1.
\end{equation*}
By invoking~\Cref{l:constants}\Cref{l:constants_coco}, we see that for each $i\in\{1,2,\ldots,n\}$, $F_i$ is $\tau'_i$-cocoercive. Consequently, we may assume without loss of generality that there is equality in \eqref{eq:taus_relation}, which we rewrite in the form
\begin{equation}\label{eq:taus_relation2}
\rho_1(\tau_1-1)+\rho_2(\tau_2-1)+\cdots+\rho_n(\tau_n-1)= 0.
\end{equation}
Conditions \eqref{eq:dc1_coco}, \eqref{eq:dc2_coco} and \eqref{eq:dc4_coco} are the same as \eqref{e:t200313a-a}, \eqref{e:t200313a-b} and \eqref{e:t200313a-d}, respectively. Thus, in order to employ \Cref{t:demi_coco} and complete the proof, it suffices to prove that \eqref{e:t200313a-c} holds. Indeed, by combining \eqref{eq:dc3_coco}, \eqref{eq:dc4_coco} and \eqref{eq:taus_relation2}, we arrive at
\begin{equation*}\label{eq:dc3_coco_bis_left}
\begin{aligned}
\sum_{i=1}^n \rho_i\big(x_{i,k}-\tau_iF_i(x_{i,k})\big)& = \sum_{i=1}^n \rho_i\big(x_{i,k}-F_i(x_{i,k})\big)+\sum_{i=1}^n \rho_i(1-\tau_i)F_i(x_{i,k})\\
& = \sum_{i=1}^n \rho_i\big(x_{i,k}-F_i(x_{i,k})\big)+\sum_{i=2}^n \rho_i(1-\tau_i)\big(F_i(x_{i,k})-F_1(x_{1,k})\big)\\
& \to -\bigg(\sum_{i=1}^n\rho_i\bigg)y+\sum_{i=1}^n\rho_ix_i=-\bigg(\sum_{i=1}^n\rho_i\tau_i\bigg)y+\sum_{i=1}^n\rho_ix_i,
\end{aligned}
\end{equation*}
which is \eqref{e:t200313a-c}.
\end{proof}

\begin{remark}[on the balance condition \eqref{eq:taus_relation}]
We note that the conditions in \eqref{eq:dc_conditions_coco} for cocoercive mappings are a weighted version of the conditions in \eqref{eq:dc_conditions_firm} for firmly nonexpansive mappings. However, the cocoercivity constants are required to be balanced as in \eqref{eq:taus_relation}; that is, the weighted average of the cocoercivity constants has to be at least $1$, which is always true for firmly nonexpansive mappings (see~\Cref{rem:constants}\Cref{rem:constants_coco}).
\end{remark}

\subsection{Demiclosedness principles for conically averaged operators}

In this section, we provide a demiclosedness principle for finitely many conically averaged operators. This is yet another generalization of the demiclosedness principle for firmly nonexpansive operators (\cref{t:demi_fne}), which we employ in our proof. 

\begin{theorem}[demiclosedness principle for conically averaged operators]\label{t:demi_avg}
\hfill For each\\ ${i\in\{1,2,\ldots,n\}}$,
let $T_i:\Hi\to\Hi$ be conically $\theta_i$-averaged where $\theta_i>0$, and let $\seq{x_{i,k}}{k}$ be a sequence in $\Hi$. Suppose that
\begin{subequations}\label{e:t200309a}
\begin{align}
\forall i\in\{1,\ldots,n\},\quad x_{i,k}\wto x_i,\label{e:t200309a-a}\\
\forall i\in\{1,\ldots,n\},\quad T_i(x_{i,k})\wto 2\theta_i y + (1-2\theta_i)x_i,\label{e:t200309a-b}\\
\sum_{i=1}^{n}\frac{x_{i,k}-T_i(x_{i,k})}{2\theta_i}
\to -ny+\sum_{i=1}^{n}x_i,\label{e:t200309a-c}\\
\forall i, j\in \{1,2,\ldots,n\},\ (x_{i,k}-x_{j,k})-
\left(\frac{x_{i,k}-T_i(x_{i,k})}{2\theta_i}-\frac{x_{j,k}-T_j(x_{j,k})}{2\theta_j}\right)\to 0.\label{e:t200309a-d}
\end{align}
\end{subequations}
Then $T_i(x_i)=2\theta_i y+(1-2\theta_i)x_i$ for all $i\in\{1,\ldots,n\}$.
\end{theorem}
\begin{proof}
For each $i\in\{1,\ldots,n\}$, set
\begin{equation*}
F_i:=\left(1-\frac{1}{2\theta_i}\right)\Id+\frac{1}{2\theta_i} T_i
=\Id -\left(\frac{\Id-T_i}{2\theta_i}\right).
\end{equation*}
Then \Cref{fact:conical_equiv}\Cref{fact:conical_equivII} and \Cref{rem:constants}\Cref{rem:constants_coni} imply that $F_i$ is firmly nonexpansive. By employing \eqref{e:t200309a-a} and \eqref{e:t200309a-b}, we see that
\begin{equation}\label{v_coni_v1}
F_i(x_{i,k})=\Big(1-\frac{1}{2\theta_i}\Big)x_{i,k}+\frac{1}{2\theta_i} T_i(x_{i,k}) \wto \Big(1-\frac{1}{2\theta_i}\Big)x_i
+ y+ \frac{1-2\theta_i}{2\theta_i}x_i = y.
\end{equation}
Next, by invoking \eqref{e:t200309a-c}, we obtain
\begin{equation}\label{v_coni_v2}
\sum_{i=1}^{n}(x_{i,k}-F_i(x_{i,k}))= \sum_{i=1}^n\frac{x_{i,k}-T_i(x_{i,k})}{2\theta_i}\to -ny+\sum_{i=1}^{n}x_i.
\end{equation}
Finally, by employing \eqref{e:t200309a-d}, we see that for all $i,j\in\{1,\ldots,n\}$,
\begin{equation}\label{v_coni_v3}
\begin{aligned}
F_i(x_{i,k})-F_j(x_{j,k})
&=\Big(1-\frac{1}{2\theta_i}\Big)x_{i,k}-\Big(1-\frac{1}{2\theta_j}\Big)x_{j,k}
+\frac{T_i(x_{i,k})}{2\theta_i}-\frac{T_j(x_{j,k})}{2\theta_j}\\
&=(x_{i,k}-x_{j,k})+\frac{T_i(x_{i,k})-x_{i,k}}{2\theta_i} -\frac{T_j(x_{j,k})-x_{j,k}}{2\theta_j}\ \to 0.
\end{aligned}
\end{equation}
Consequently, in view of \eqref{e:t200309a-a}, \eqref{v_coni_v1}, \eqref{v_coni_v2} and \eqref{v_coni_v3}, we apply \cref{t:demi_fne} to obtain
\begin{equation*}
y=F_i(x_i)=\Big(1-\frac{1}{2\theta_i}\Big)x_i +\frac{1}{2\theta_i}T_i(x_i),\quad \forall i\in\{1,\ldots,n\},
\end{equation*}
which concludes the proof.
\end{proof}

\begin{remark}[\cref{t:demi_avg} versus \cref{t:demi_fne}]
Since firmly nonexpansive operators are conically $\theta$-averaged with $\theta=\frac{1}{2}$, it is clear that the assertion of \cref{t:demi_avg} is more general than the one of \cref{t:demi_fne}. However, in view of the proof of \cref{t:demi_avg}, we conclude that the two assertions are equivalent.
\end{remark}

We proceed in a similar manner to our discussion of demiclosedness principles for cocoercive operators, namely, we would like to have convergence conditions as in \eqref{e:t200309a} of \Cref{t:demi_avg} that do not incorporate the conical average constants $\theta_i$. Indeed, in the following result, we provide such conditions while, yet again, imposing a balance condition on the $\theta_i$'s. We focus our attention on a result concerning two mappings, which we will use in applications.

\begin{theorem}[demiclosedness principle for two balanced averaged operators]\label{th:Demiclosed_avg}
Let ${T_1,T_2:\Hi\to\Hi}$ be $\theta_1$- and $\theta_2$-averaged mappings where $\theta_1,\theta_2\in\left]0,1\right[$, respectively, and suppose that there exist scalars $\rho_1,\rho_2>0$ such that
\begin{equation}\label{eq:thetas_relation}
\theta_1\leq\frac{\rho_2}{\rho_1+\rho_2} \quad\text{and}\quad \theta_2\leq\frac{\rho_1}{\rho_1+\rho_2}.
\end{equation}
Let $(x_{1,k})_{k=0}^{\infty}$ and $(x_{2,k})_{k=0}^{\infty}$ be sequences in $\Hi$ such that 
\begingroup
\allowdisplaybreaks
\begin{subequations}\label{eq:dc_conditions_conical}
 	\begin{align}
 	x_{1,k}\wto x_1 \quad\text{ and }\quad x_{2,k}\wto x_2,& \label{eq:dc1_conical} \\
 	T_1(x_{1,k})\wto y \quad\text{ and }\quad T_2(x_{2,k})\wto y,& \label{eq:dc2_conical} \\
 	\rho_1(x_{1,k}-T_1(x_{1,k}))+\rho_2(x_{2,k}-T_2(x_{2,k}))\to 0,& \label{eq:dc3_conical} \\
 	T_1(x_{1,k})-T_2(x_{2,k})\to 0.& \label{eq:dc4_conical}
 	\end{align}
\end{subequations}
\endgroup
Then $T_1(x_1)=T_2(x_2)=y$.
\end{theorem}
\begin{proof}
As in the proof of \Cref{th:Demiclosed_coco}, by employing ~\Cref{l:constants}\Cref{l:constants_conical}, we may assume without the loss of generality that there is equality in \eqref{eq:thetas_relation}, which we rewrite in the form
\begin{subequations}\label{eq:thetas_relation2}
\begin{align}
\left( \tfrac{1}{2}-\theta_1 \right) + \left( \tfrac{1}{2}-\theta_2 \right) = 0,\label{eq:thetas_relation2I}\\
\text{and}\quad\rho_1\theta_1=\rho_2\theta_2.\label{eq:thetas_relation2II}
\end{align}
\end{subequations}
By combining  \eqref{eq:dc1_conical},  \eqref{eq:dc2_conical}  and \eqref{eq:dc3_conical} we see that $\rho_1(x_1-y)+\rho_2(x_2-y)=0$, equivalently,
\begin{equation}\label{eq:y0}
y=\frac{\rho_1x_1+\rho_2x_2}{\rho_1+\rho_2}=\theta_2x_1+\theta_1x_2.
\end{equation}
We set $\overline{y}:=\frac{1}{2}(x_1+x_2)$. By invoking \eqref{eq:thetas_relation2I} it follows that
\begin{equation}\label{eq:y00}
\begin{aligned}
2\theta_1\overline{y}+(1-2\theta_1)x_1=(1-\theta_1)x_1+\theta_1x_2=\theta_2x_1+\theta_1x_2,\\
2\theta_2\overline{y}+(1-2\theta_2)x_2=\theta_2x_1+(1-\theta_2)x_2=\theta_2x_1+\theta_1x_2.
\end{aligned}
\end{equation}
Consequently, \eqref{eq:dc2_conical} and \eqref{eq:y0} imply that
\begin{equation}\label{eq:th_conical_newFwto}
\begin{aligned}
T_1(x_{1,k}) & \wto 2\theta_1\overline{y}+(1-2\theta_1)x_1,\\
T_2(x_{2,k}) & \wto 2\theta_2\overline{y}+(1-2\theta_2)x_2.
\end{aligned}
\end{equation}
Now, by \eqref{eq:thetas_relation2II}, \eqref{eq:dc3_conical} and the definition of $\overline{y}$, it follows that
\begin{equation}\label{eq:th_conical_sumtosum}
\begin{aligned}
&\frac{x_{1,k}-T_1(x_{1,k})}{2\theta_1}+\frac{x_{2,k}-T_2(x_{2,k})}{2\theta_2}\\&\qquad=\frac{1}{2\rho_2\theta_2}\big(\rho_1\left(x_{1,k}-T_1(x_{1,k})\right)+\rho_2\left(x_{2,k}-T_2(x_{2,k})\right)\big)\to 0=-2\overline{y}+x_1+x_2.\\
\end{aligned}
\end{equation}
In addition, from \eqref{eq:thetas_relation2I} it follows that
\begin{equation}\label{e:200721a}
\begin{aligned}
(x_{1,k}&-x_{2,k})-\left(\frac{x_{1,k}-T_1(x_{1,k})}{2\theta_1}-\frac{x_{2,k}-T_2(x_{2,k})}{2\theta_2}\right)=\\
&=\frac{T_1(x_{1,k})-(1-2\theta_1)x_{1,k}}{2\theta_1}+\frac{(1-2\theta_2)x_{2,k}-T_2(x_{2,k})}{2\theta_2}\\
&=\frac{(1-2\theta_1)}{2\theta_1}\big(T_1(x_{1,k})-x_{1,k}\big)+T_1(x_{1,k})-T_2(x_{2,k})+\frac{(1-2\theta_2)}{2\theta_2}\big(x_{2,k}-T_2(x_{2,k})\big)\\
&=(1-2\theta_2)\left( \frac{x_{1,k}-T_1(x_{1,k})}{2\theta_1}+\frac{x_{2,k}-T_2(x_{2,k})}{2\theta_2}\right)+\big(T_1(x_{1,k})-T_2(x_{2,k})\big).
\end{aligned}
\end{equation}
By combining \eqref{e:200721a} with \eqref{eq:th_conical_sumtosum} and \eqref{eq:dc4_conical}, we obtain
\begin{equation}\label{eq:th_conical_difto0}
(x_{1,k}-x_{2,k})-\left(\frac{x_{1,k}-T_1(x_{1,k})}{2\theta_1}-\frac{x_{2,k}-T_2(x_{2,k})}{2\theta_2}\right)\to 0.
\end{equation}
Finally, in view of \eqref{eq:dc1_conical}, \eqref{eq:th_conical_newFwto},  \eqref{eq:th_conical_sumtosum} and \eqref{eq:th_conical_difto0}, we employ \Cref{t:demi_avg} in order to obtain
\begin{equation*}
T_1(x_1)=2\theta_1\overline{y}+(1-2\theta_1)x_1\quad\text{and}\quad
T_2(x_2)=2\theta_2\overline{y}+(1-2\theta_2)x_2.
\end{equation*}
By recalling \eqref{eq:y0} and \eqref{eq:y00}, we arrive at $T_1(x_1)=T_2(x_2)=y$.
\end{proof}

\begin{remark}[on the balance condition]
In \Cref{th:Demiclosed_avg}, we did not provide an explicit balance condition for the averagedness constants as we did in~\eqref{eq:taus_relation}. However, we did impose a stronger condition~\eqref{eq:thetas_relation}, which indeed implies
\begin{equation*}
\frac{\rho_1\theta_1+\rho_2\theta_2}{\rho_1+\rho_2}\leq \frac{2\rho_1\rho_2}{(\rho_1+\rho_2)^2}\leq \frac{1}{2};
\end{equation*}
that is, the weighted average of the $\theta_i$'s is at most $\frac{1}{2}$. This is always true for firmly nonexpansive mappings (see~\Cref{rem:constants}\Cref{rem:constants_coni}).
\end{remark}

\section{Applications to the adaptive Douglas--Rachford algorithm}\label{sec:aDR}

We recall that the problem of finding a zero of the sum of two operators $A,B:\Hi\tto\Hi$ is
\begin{equation}\label{e:0inA+B}
\text{find } x\in\Hi
\text{ such that } 0 \in Ax+Bx.
\end{equation}

In this section, we apply our generalized demiclosedness principles and derive the {\em weak convergence} of the shadow sequence of the \emph{adaptive Douglas--Rachford (aDR)} algorithm, originally introduced in \cite{DPadapt}, in order to solve \eqref{e:0inA+B} for a weakly and a strongly monotone operators. The analysis is then extended in \cite{BDP19}, which includes weakly comonotone and strongly comonotone operators as well.

Given $(\gamma,\delta,\lambda,\mu,\kappa)\in\RR_{++}^5$, the {\em aDR operator} is defined by
\begin{equation}\label{eq:aDR_opeartor}
\Tdr:=(1-\kappa)\Id+\kappa R_2R_1,
\end{equation}
where
\begin{equation*}
\begin{aligned}
J_1&:=J_{\gamma A}=(\Id+\gamma A)^{-1},
& R_1&:=J^{\lambda}_{\gamma A}=(1-\lambda)\Id+\lambda J_1,\\
J_2&:=J_{\delta B}=(\Id+\delta B)^{-1},
& R_2&:=J^{\mu}_{\delta A}=(1-\mu)\Id+\mu J_2.
\end{aligned}
\end{equation*}
Set an initial point $x_0\in\Hi$. The aDR algorithm generates a sequence $\seq{x_k}{k}$ by the recurrence
\begin{equation}\label{e:TaDR}
x_{k+1}\in\Tdr(x_k),\quad  k=0,1,2,\ldots.
\end{equation}
We observe that~\eqref{eq:aDR_opeartor} coincides with the classical Douglas--Rachford operator in the case where $\lambda=\mu:=2$, $\gamma=\delta $, and $\kappa=1/2$. Similarly to the classical DR algorithm, the fixed points of $\Tdr$ are not explicit solutions of \eqref{e:0inA+B}. Nonetheless, under the assumptions (see~\cite[Section~5]{BDP19}) 
\begin{subequations}\label{e:lm_mu}
\begin{align}
(\lambda-1)(\mu-1)=1 \quad\text{and}\quad \delta=(\lambda-1)\gamma,
\end{align}
equivalently,
\begin{align}
\lambda=1+\frac{\delta}{\gamma}
\quad\text{and}\quad
\mu=1+\frac{\gamma}{\delta},
\end{align}
\end{subequations}
the fixed points are useful in order to obtain a solution as we show next.

\begin{fact}
[aDR and solutions to the inclusion problem]
\label{f:aDR_fix}
Suppose that $(\gamma,\delta)\in\RR_{++}^2$, that $\lambda,\mu$ are defined by \eqref{e:lm_mu}, that $\kappa>0$, and that $J_1$ is single-valued. Then
\begin{enumerate}[label={\rm(\roman*)}]
\item $\Id-\Tdr=\kappa\mu(J_1-J_2R_1)$;\label{f:aDR_fix_i}
\item $J_1(\Fix\Tdr)=\zer(A+B)$.\label{f:aDR_fix_ii}
\end{enumerate}
\end{fact}
\begin{proof}
See~\cite[Lemma~4.1]{DPadapt}.
\end{proof}

Suppose that $\seq{x_k}{k}$ is generated by the aDR algorithm and converges weakly to the limit point $x^\star\in\Fix \Tdr$. Then \cref{f:aDR_fix}\ref{f:aDR_fix_ii} asserts that the \emph{shadow} limit point $J_1(x^\star)$ is a solution of \eqref{e:0inA+B}. Our aim is to prove that, under certain assumptions, the shadow sequence $\seq{J_1(x_k)}{k}$ converges weakly to the shadow limit~$J_1(x^\star)$.

In our analysis we will employ the convergence results \cite[Theorems~5.4 and 5.7]{BDP19}: Under the assumptions therein, the aDR operator \eqref{e:TaDR} is shown to be averaged and, hence, singled-valued. Consequently, in these cases, we will employ equality in \eqref{e:TaDR}.

\subsection{Adaptive DR algorithm for monotone operators}

We begin our discussion with the case where the operators are maximally $\alpha$-monotone and maximally $\beta$-monotone. We will prove the weak convergence of the aDR algorithm shadow sequence by means of a generalized demiclosedness principle. To this end, we recall the following fact regarding the convergence of the aDR algorithm.

\begin{fact}[aDR for monotone operators]\label{fact:aDR_mono}
Let $\alpha,\beta\in\RR$ such that $\alpha+\beta\geq 0$ and let $A,B:\Hi\tto\Hi$ be maximally $\alpha$-monotone and maximally $\beta$-monotone, respectively,  with $\zer(A+B)\neq \varnothing$. Let $(\gamma,\delta,\lambda,\mu)\in\RR^4_{++}$ satisfy \eqref{e:lm_mu} and either
\begingroup
\allowdisplaybreaks
\begin{subequations}\label{eq:assumtions_mono}
\begin{align}
\delta(1+2\gamma\alpha)=\gamma, &\quad \text{ if } \alpha+\beta=0,\\
\text{or } \quad (\gamma+\delta)^2<4\gamma\delta(1+\gamma\alpha)(1+\delta\beta), &\quad \text{ if } \alpha+\beta>0.
\end{align}
\end{subequations}
\endgroup
Set $\kappa\in{]0,\overline{\kappa}[}$ where 
\begin{equation}\label{eq:kappa_mono}
\overline{\kappa}:=\left\{\begin{array}{ll}
1, & \text{if } \alpha+\beta=0;\\
\displaystyle\frac{4\gamma\delta(1+\gamma\alpha)(1+\delta\beta)-(\gamma+\delta)^2}{2\gamma\delta(\gamma+\delta)(\alpha+\beta)}, & \text{if } \alpha+\beta>0.
\end{array}\right.
\end{equation}
Set a starting point $x_0\in\Hi$ and $(x_k)_{k=0}^{\infty}$ by $x_{k+1}=\Tdr(x_k),\ k=0,1,2,\ldots.$ Then
\begin{enumerate}[label={\rm(\roman*)}]
\item $x_k-x_{k+1}\to 0;$ \label{fact:aDR_monoI}
\item $x_k\wto x^\star\in \Fix \Tdr \text{ with } J_1(x^\star)\in \zer(A+B).$  \label{fact:aDR_monoII}
\end{enumerate}
\end{fact}
\begin{proof}
Combine~\cite[Theorem 5.7]{BDP19} with \cite[Corollary~2.10]{BDP19} and \Cref{f:aDR_fix}\Cref{f:aDR_fix_ii}.
\end{proof}

\begin{theorem}[weak convergence of the shadow sequence under monotonicity]\label{th:aDR_shadow_mono}
\	\\
Suppose that $A$ and $B$ are maximally $\alpha$-monotone and maximally $\beta$-monotone, respectively, where $\alpha+\beta\geq 0$ and $\zer(A+B)\neq \varnothing$. Let $(\gamma,\delta,\lambda,\mu)\in\RR_{++}^4$ satisfy \eqref{e:lm_mu} and \eqref{eq:assumtions_mono}. Let $\kappa\in {]0,\overline{\kappa}[}$ where $\overline{\kappa}$ is defined by \eqref{eq:kappa_mono}. Set a starting point $x_0\in\Hi$ and $(x_k)_{k=0}^{\infty}$ by
$x_{k+1}=\Tdr(x_k),\ k=0,1,2,\ldots.$
Then the shadow sequence $\seq{J_1(x_k)}{k}$ converges weakly
$$J_1(x_k)\wto J_1(x^\star)\in \zer(A+B).$$
\end{theorem}
\begin{proof}
\Cref{fact:aDR_mono}\Cref{fact:aDR_monoII} asserts that
\begin{equation}\label{eq:aDR_coco_cond1a}
x_k\wto x^\star \in \Fix \Tdr \text{ with } J_1(x^\star)\in \zer(A+B).
\end{equation}
Consequently, $(x_k)_{k=0}^{\infty}$ is bounded. By combining \eqref{e:lm_mu} and \eqref{eq:assumtions_mono} it follows that $1+\gamma\alpha>0$ and $1+\delta\beta>0$ (see \cite[Theorem~5.7]{BDP19}). Thus, we employ \Cref{fact:resolvent_mono} which asserts that $J_1$ and $J_2$ are $\tau_1$-cocoercive and $\tau_2$-cocoercive, respectively, with full domain, where
$$\tau_1:=1+\gamma\alpha \quad\text{and}\quad \tau_2:=1+\delta\beta.$$
Due to the cocoerciveness of $J_1$, the shadow sequence $\left(J_1(x_k)\right)_{k=0}^{\infty}$ is bounded and has a weak converging subsequence, say,
\begin{equation}\label{eq:aDR_coco_cond2a}
J_1\left(x_{k_j}\right)\wto y^\star.
\end{equation}
Set $z_k:= R_1(x_k)$, for each $k=0,1,2,\ldots$. Then
\begin{equation}\label{eq:aDR_coco_cond1b}
z_{k_j} \wto (1-\lambda)x^\star+ \lambda y^\star=:z^\star.
\end{equation}
Moreover, \Cref{fact:aDR_mono}\Cref{fact:aDR_monoI} and \Cref{f:aDR_fix}\Cref{f:aDR_fix_i} imply that
\begin{equation}\label{eq:aDR_coco_cond3}
J_1(x_k) - J_2(z_k)\to 0.
\end{equation}
which, when combined with \eqref{eq:aDR_coco_cond2a}, implies that
\begin{equation}\label{eq:aDR_coco_cond2b}
J_2(z_{k_j})\wto y^\star.
\end{equation}
Thus, on the one hand $J_1(x_{k_j}) - J_2(z_{k_j})\to 0$ while, on the other hand,
\begin{equation}\label{eq:aDR_coco_cond4}
\begin{aligned}
J_1(x_{k_j}) - J_2(z_{k_j})
& =  J_1(x_{k_j}) - R_1(x_{k_j})+ R_1(x_{k_j}) - J_2(z_{k_j})\\
& = (\lambda-1)\left( x_{k_j}- J_{1}(x_{k_j}) \right) + \left(z_{k_j}- J_{2}(z_{k_j})\right)\\
& \wto (\lambda-1)(x^\star-y^\star)+(z^\star - y^\star)= -\lambda y^\star +(\lambda-1)x^\star +z^\star.
\end{aligned}
\end{equation}
By combining \eqref{eq:aDR_coco_cond1a}--\eqref{eq:aDR_coco_cond4}, we see that the sequences $(x_{k_j})_{j=0}^{\infty}$ and $(z_{k_j})_{j=0}^{\infty}$ satisfy the conditions in \eqref{eq:dc_conditions_coco} by setting
\begin{equation}\label{eq:choice_rhos}
\rho_1:=\lambda-1>0 \quad\text{and}\quad  \rho_2:=1>0.
\end{equation}
With this choice of the parameters $\rho_i$'s,
we observe that the balance condition \eqref{eq:taus_relation} is satisfied as well. Indeed, \eqref{e:lm_mu} implies that
\begin{equation*}\label{eq:rel_taus_aDR}
\rho_1(\tau_1-1)+\rho_2(\tau_2-1)=(\lambda-1)\gamma\alpha+\delta\beta=\delta(\alpha+\beta)\geq 0.
\end{equation*}
Consequently, we apply \Cref{th:Demiclosed_coco} in order to obtain $y^\star=J_1(x^\star).$
\end{proof}

\begin{remark}
We observe that \Cref{th:aDR_shadow_mono} guarantees the weak convergence of the shadow sequence whenever the original sequence converges weakly. In particular, this is guaranteed under the conditions on the parameters in \Cref{fact:aDR_mono}. However, this is a meaningful contribution only in the case where $\alpha+\beta=0$: In the case where $\alpha+\beta>0$, it is known that the shadow sequence converges not only weakly but, in fact, strongly (see~\cite[Theorem~4.5]{DPadapt} and \cite[Remark~5.8]{BDP19}). 
\end{remark}

\subsection{Adaptive DR algorithm for comonotone operators}

We now address the weak convergence of the shadow sequence in the case where the operators are comonotone. To this end, we recall the following result regarding the convergence of the aDR algorithm.

\begin{fact}[aDR for comonotone operators]\label{fact:aDR_comono}
Let $\alpha,\beta\in\RR$ such that $\alpha+\beta\geq 0$ and let $A,B:\Hi\tto\Hi$ be maximally $\alpha$-comonotone and maximally $\beta$-comonotone, respectively, such that $\zer(A+B)\neq \varnothing$. Let $(\gamma,\delta,\lambda,\mu)\in\RR^4_{++}$ satisfy \eqref{e:lm_mu} and either
\begingroup
\allowdisplaybreaks
\begin{subequations}\label{eq:assumtions_comono}
\begin{align}
\delta=\gamma+2\alpha, &\quad \text{ if } \alpha+\beta=0,\label{eq:assumtions_comonoI}\\
\text{or } \quad (\gamma+\delta)^2<4(\gamma+\alpha)(\delta+\beta), &\quad \text{ if } \alpha+\beta>0.\label{eq:assumtions_comonoII}
\end{align}
\end{subequations}
\endgroup
Set $\kappa\in{]0,\overline{\kappa}[}$ where
\begin{equation}\label{eq:kappa_comono}
\overline{\kappa}:=\left\{\begin{array}{ll}
1, & \text{if } \alpha+\beta=0;\\
\displaystyle\frac{4(\gamma+\alpha)(\delta+\beta)-(\gamma+\delta)^2}{2(\gamma+\delta)(\alpha+\beta)}, & \text{if } \alpha+\beta>0.
\end{array}\right.
\end{equation}
Set a starting point $x_0\in\Hi$ and $(x_k)_{k=0}^{\infty}$ by $x_{k+1}=\Tdr(x_k),\ k=0,1,2,\ldots.$ Then
\begin{enumerate}[label={\rm(\roman*)}]
\item $x_k-x_{k+1}\to 0;$ \label{fact:aDR_comonoI}
\item $x_k\wto x^\star\in \Fix \Tdr \text{ with } J_1(x^\star)\in \zer(A+B).$  \label{fact:aDR_comonoII}
\end{enumerate}
\end{fact}
\begin{proof}
Combine~\cite[Theorem 5.4]{BDP19} with \cite[Corollary~2.10]{BDP19} and \Cref{f:aDR_fix}\Cref{f:aDR_fix_ii}.
\end{proof}

\begin{theorem}[weak convergence of the shadow sequence under comonotonicity]\label{th:aDR_shadow_comono}
Suppose that $A$ and $B$ are maximally $\alpha$-comonotone and maximally $\beta$-comonotone such that $\alpha+\beta\geq 0$ and $\zer(A+B)\neq \varnothing$. Let $(\gamma,\delta,\lambda,\mu)\in\RR_{++}^4$ satisfy \eqref{e:lm_mu} and suppose that
\begin{equation}\label{eq:assumtions_comono_restricted}
(\gamma+\delta) \leq \min\{2(\gamma+\alpha),2(\delta+\beta)\}.
\end{equation}
Let $\kappa\in {]0,\overline{\kappa}[}$ where $\overline{\kappa}$ is defined by \eqref{eq:kappa_comono}. Set a starting point $x_0\in\Hi$  and $(x_k)_{k=0}^{\infty}$ by
$x_{k+1}=\Tdr(x_k),\ k=0,1,2,\ldots.$
Then the shadow sequence $\seq{J_1(x_k)}{k}$ converges weakly 
$$J_1(x_k)\wto J_1(x^\star)\in \zer(A+B).$$
\end{theorem}
\begin{proof}
We claim that \eqref{eq:assumtions_comono_restricted} implies \eqref{eq:assumtions_comono}. Indeed, if $\alpha+\beta=0$, we obtain
\begin{align*}
(\gamma+\delta) \leq \min\{2(\gamma+\alpha),2(\delta-\alpha)\} & 
\ \iff\ (\gamma+\delta)\leq 2(\gamma+\alpha) \quad\text{and}\quad (\gamma+\delta)\leq 2(\delta-\alpha)\\
& \ \iff\ \delta = \gamma+2\alpha,
\end{align*}
which is \eqref{eq:assumtions_comonoI}. On the other hand, observe that \eqref{eq:assumtions_comono_restricted} trivially implies that
\begin{equation*}
(\gamma+\delta)^2\leq (\min\{2(\gamma+\alpha),2(\delta+\beta)\})^2 \leq 4(\gamma+\alpha)(\delta+\beta).
\end{equation*}
Suppose that $(\gamma+\delta)^2=(\min\{2(\gamma+\alpha),2(\delta+\beta)\})^2 = 4(\gamma+\alpha)(\delta+\beta)$. Then
\begin{equation*}
(\gamma+\delta)=2(\gamma+\alpha)=2(\beta+\delta),
\end{equation*}
which implies that $\alpha+\beta=0$. Thus, \eqref{eq:assumtions_comonoII} holds whenever $\alpha+\beta>0$ which concludes the proof of our claim. Consequently, we employ \Cref{fact:aDR_comono} in order to obtain 
\begin{equation*}
x_k\wto x^\star \in \Fix \Tdr \text{ with } J_1(x^\star)\in \zer(A+B).
\end{equation*}
We conclude that $(x_k)_{k=0}^{\infty}$ is bounded. Now, \eqref{e:lm_mu} and \eqref{eq:assumtions_comono} imply that $\gamma+\alpha>0$ and $\delta+\beta>0$ (see \cite[Theorem~5.4]{BDP19}). Thus, by \Cref{fact:resolvent_comono}, $J_1$ and $J_2$ are conically $\theta_1$-averaged and $\theta_2$-averaged, respectively, with full domain, where
\begin{equation}\label{eq:conical_resolvents}
\theta_1:= \frac{\gamma}{2(\gamma+\alpha)} \quad\text{and}\quad \theta_2:= \frac{\delta}{2(\delta+\beta)}.
\end{equation}
Since $J_1$ is conically averaged, the shadow sequence $\left(J_1(x_k)\right)_{k=0}^{\infty}$ is bounded and has a weakly convergent subsequence, say, $J_1(x_{k_j})\rightharpoonup y^\star$.  By the same arguments as in the proof of \Cref{th:aDR_shadow_mono}, while employing \Cref{th:Demiclosed_avg} instead of \Cref{th:Demiclosed_coco}, we arrive at
$$y^\star=J_1(x^\star),$$
which concludes the proof. To this end, it remains to verify that \eqref{eq:thetas_relation} holds and, then, \Cref{th:Demiclosed_avg} is applicable. Indeed, by~\eqref{eq:conical_resolvents}, \eqref{eq:choice_rhos} and \eqref{e:lm_mu}, 
\begin{align*}
\theta_1\leq\frac{\rho_2}{\rho_1+\rho_2} \quad\text{and}\quad \theta_2\leq\frac{\rho_1}{\rho_1+\rho_2} & 
\quad\iff\quad  \frac{\gamma}{2(\gamma+\alpha)}\leq\frac{1}{\lambda} \quad\text{and}\quad \frac{\delta}{2(\delta+\beta)}\leq\frac{\lambda-1}{\lambda}\\
& \quad\iff\quad  \frac{\gamma}{2(\gamma+\alpha)}\leq\frac{1}{\lambda} \quad\text{and}\quad \frac{\gamma}{2(\delta+\beta)}\leq\frac{1}{\lambda} \\[1.5ex]
& \quad\iff\quad \gamma\lambda \leq \min\{2(\gamma+\alpha),2(\delta+\beta)\}\\[1.5ex]
& \quad\iff\quad (\gamma+\delta) \leq \min\{2(\gamma+\alpha),2(\delta+\beta)\},
\end{align*}
which is \eqref{eq:assumtions_comono_restricted}.
\end{proof}

\begin{remark}
We note that, in contrast to \Cref{th:aDR_shadow_mono}, \Cref{th:aDR_shadow_comono} guarantees the weak convergence of the shadow sequence under stronger conditions on the parameters than the conditions in \Cref{fact:aDR_comono}, namely, in the case where $\alpha+\beta>0$. Nevertheless, \Cref{th:aDR_shadow_comono} covers new ground since the convergence of the shadow sequence in the aDR for comonotone operators has not been previously addressed. Although outside the scope of this work, we believe that the convergence of the shadow sequence may be strong whenever $\alpha+\beta>0$, as in the case of monotone operators.
\end{remark}

\section{Conclusions}\label{sec:conclu}

In this paper, we extend the multi-operator demiclosedness principle \cite{Demiclosed} to more general classes of operators such as cocoercive and conically averaged operators. The new findings are natural and consistent with existing theory, and are later justified by applications in which we show the weak convergence of the shadow sequence of the adaptive Douglas--Rachford algorithm. It remains of interest to find new connections between the demiclosedness principle and other classes of algorithms and optimization problems.

\renewcommand{\baselinestretch}{1}

{\small
\paragraph{\small Acknowledgements}
Sedi Bartz was partially supported by a UMass Lowell faculty startup grant. Rub\'en Campoy was partially supported by a postdoctoral fellowship of UMass Lowell. Hung M. Phan was partially supported by Autodesk, Inc. via a gift made to the Department of Mathematical Sciences,
UMass Lowell.}

\end{document}